\documentclass[sts]{imsart}

\usepackage{times}

\usepackage{amssymb,amsbsy,amsmath,amscd,amsthm,amsfonts,MnSymbol}
\usepackage{cite}
\usepackage[utf8]{inputenc}
\usepackage{enumerate}
\usepackage{hyperref}
\usepackage{dsfont}
\usepackage{graphicx}
\usepackage{enumitem}
\usepackage{wrapfig}

\usepackage{algorithm}
\usepackage{algorithmic}

\newtheorem{theorem}{Theorem}

\newtheorem{lemma}{Lemma}
\newtheorem{remark}{Remark}

\newtheorem{corollary}{Corollary}

\newcommand{\R}{\mathbb R}

\newcommand{\PP}{\mathbb P}
\newcommand{\QQ}{\mathbb Q}
\newcommand{\E}{\mathbb E}

\newcommand{\Sp}{\mathbb S}

\newcommand{\Sig}{\mathfrak{S}_d}

\newcommand{\Var}{\textsf{Var}}

\newcommand{\cov}{\textsf{cov}}

\newcommand{\Tr}{\textsf{Tr}}

\newcommand*\diff{\mathop{}\!\mathrm{d}}

\newcommand{\DS}{\displaystyle}

\newcommand{\prt}[1]{\left( \, #1  \, \right)}
\newcommand{\crc}[1]{\left[ \, #1  \, \right]}

\newcommand{\pl}[1]{\partial ^#1 \ell}
\newcommand{\pll}{\partial \ell}

\begin{document}

\begin{frontmatter}

\title{Learning rates for Gaussian mixtures under group invariance}
\runtitle{Gaussian mixtures with isometry group invariance}


\author{\fnms{Victor-Emmanuel} \snm{Brunel}\ead[label=e1]{victor.emmanuel.brunel@ensae.fr}}
\address{\printead{e1}}
\affiliation{ENSAE ParisTech}

\runauthor{V.-E. Brunel}

\setattribute{abstractname}{skip} {{\bf Abstract:} }

\begin{abstract}
We study the pointwise maximum likelihood estimation rates for a class of Gaussian mixtures that are invariant under the action of some isometry group. This model is also known as multi-reference alignment, where random isometries of a given vector are observed, up to Gaussian noise. We completely characterize the speed of the maximum likelihood estimator, by giving a comprehensive description of the likelihood geometry of the model. We show that the unknown parameter can always be decomposed into two components, one of which can be estimated at the fast rate $n^{-1/2}$, the other one being estimated at the slower rate $n^{-1/4}$. We provide an algebraic description and a geometric interpretation of these facts.

\end{abstract}



\end{frontmatter}

\section{Introduction}

In practical situations, when one has access to many noisy observations of an object, that object may have been rotated, or shifted, across the observations. This can be the case, for instance, in chemistry or nanobiology: If the goal is to learn the structure of a molecule from many samples, the molecule is very likely to move or, even, to appear as one of its isomers, in each sample. Then, the main challenge, on top of denoising the data, is to align all the observations together. When the configuration of the unknown object in each observation is itself random, the observation scheme can be modeled as a mixture of distributions, where each component of the mixture is centered around a modified version of the unknown object. When these versions are all isometric transformations of each other, the problem is also called the multi-reference alignment problem, see \cite{sorzano2010clustering,singer2011three,bandeira2014multireference,bandeira2017optimal,perry2017sample,wein2018statistical} and the references therein. To fix the ideas, we formalize the model as follows. Let the unknown object be represented by a vector $\theta^*\in\R^d$ and assume that we have access to $n$ independent observations $y_i=g_i\theta^*+\sigma\varepsilon_i, i=1,\ldots,n$ where $g_i\in G$ is possibly random, $G$ is a finite subgroup of isometries, $\sigma>0$ is known and $\varepsilon_i$ is a standard Gaussian vector. Here, we assume that $\sigma$ is known, in order to simplify the exposition: Our focus is only on understanding the challenges in learning $\theta^*$ and we might as well assume that $\sigma^2=1$, for the sake of simplicity. Here, we let $G$ be finite, but we believe that our results can be easily extended to the case of any closed (therefore compact) subgroup of isometries. We focus on the Gaussian noise setup because when the $g_i$'s are i.i.d., the model amounts to a mixture of Gaussian distributions, which is an extremely important model in modern learning theory and it still generates very active research, see e.g. \cite{dasgupta1999learning,sanjeev2001learning,kalai2010efficiently,moitra2010settling,azizyan2013minimax,hsu2013learning,hardt2015tight,amendola2016algebraic,balakrishnan2017statistical,wu2018optimal} and the references therein. Moreover, by multiplying each observation $y_i$ by an independent element of the group $G$, chosen uniformly at random, one can always assume that our observations come from a Gaussian mixture with uniform weights, which we assume in the sequel. However, we believe that in that setup, our results could be extended to mixtures of more general families with a location parameter.

For $\theta\in\R^d$, denote by $\PP_\theta$ the mixture of Gaussians with means $g\theta, g\in G$, identity covariance matrix and uniform weights, that is, $\DS \PP_\theta=\frac{1}{|G|}\sum_{g\in G}\mathcal N(g\theta,I)$. In this model, $\theta$ is not identified since $\PP_\theta=\PP_{g\theta}$, for all $g\in G$. Hence, $\theta$ can only be estimated up to the action of the group $G$. It is important to note that the bigger $G$ is, the less information the model carries about $\theta$. Consider the following two extremes: When $G=\{I\}$ and when $G=\mathcal O(d)$ (the group of all isometries). In the former case, $\theta$ is fully identified, hence, all its $d$ components can be learned. In the latter case, only the Euclidean norm of $\theta$ is identified, which is a one dimensional parameter and the estimation of $\|\theta\|$ becomes a much easier problem.

There are two most popular strategies for learning mixtures of Gaussians: Maximum likelihood estimation and methods of moments. The maximum likelihood estimator (MLE), which we focus on in this work, is usually implemented via the expectation-maximization (EM) algorithm. We refer to \cite{balakrishnan2017statistical} who recently showed some asymptotic guarantees for the EM algorithm by analysing its behavior at the population level. More generally, we emphasize the importance of understanding a statistical model in depth at a population level (which amounts to studying its asymptotics as the sample size grows to infinity), which is what motivates our work. The method of moments is algorithmically more feasible, with more algorithmic guarantees, see \cite{wu2018optimal} and the references therein. 

In this work, we are interested in pointwise rates for the estimation of the parameter $\theta^*$, i.e., the estimation of the centers of the mixture $\PP_{\theta^*}$. In \cite{chen1995optimal, heinrich2015optimal, ho2016convergence}, pointwise rates are obtained (together with minimax rates) for the estimation of $\PP_{\theta^*}$, which are similar to ours ($n^{-1/2}$ and $n^{-1/4}$). They measure the accuracy of their estimators in terms of distances between distributions (e.g., Hellinger, or Wasserstein metrics, the latter proving to be a natural choice for mixtures because of the lack of identifiability \cite{nguyen2013convergence}). However in practice, these metrics between distributions do not easily translate into a geometric distance between their parameters, hence, in our setup, they can not provide a subtle enough description of the pointwise estimation rates for the centers of the mixture. The main difficulty, in general, is that there is no natural metric for the parameter space due to the lack of identifiability of the parameters. Here, the mixtures exhibit a specific structure associated with the group $G$ and the identifiable set for $\theta^*$ (i.e., the collection of all vectors $\theta$ such that $\PP_{\theta}=\PP_{\theta^*}$) is $\Theta(\theta^*)=\{g\theta^*:g\in G\}$. Hence, there is a natural metric on the identifiable sets which translates into a geometric metric between the centers, namely, $\tilde\rho(\Theta(\theta),\Theta(\theta'))=\min_{t\in\Theta(\theta),t'\in\Theta(\theta')}\|t-t'\|=\min_{g\in G}\|g\theta-\theta'\|$. Thus, we can measure the learning error in terms of the Euclidean norm in the parameter space, which allows us to break down our analysis to the individual rates for each component of $\theta^*$ separately. By this, we mean that we can show that some components of $\theta^*$ can be estimated at a given rate, whereas other components of the same $\theta^*$ may be estimated at a faster rate (this will be made precise in Theorem \ref{ThmStats}). In \cite{bandeira2017optimal,perry2017sample}, the focus is on the minimax rates for the estimation of $\theta^*$, only when $G$ is the group of coordinate cyclic shifts. Interestingly, as already pointed out by \cite{heinrich2015optimal}, there may be a huge discrepancy between pointwise and minimax rates, due to the non uniformity of the pointwise rates. In \cite{bandeira2017optimal} and \cite{perry2017sample}, it is assumed that all the centers $g\theta^*, g\in G$ are separated away from each other, which yields $n^{-1/2}$ rates: There, the focus is rather on the dependence of the optimal rates on $\sigma^2$, which matters a lot in applications where the signal-to-noise ratio can be very small, e.g., cryo-elctron microscopy. However, imposing that all the centers are pairwise distinct can be interpreted as assuming that $\theta^*$ does not exhibit any symmetry that would be encoded in $G$, which, in practice, is debatable.

In Section \ref{SectionLikelihood}, we describe the likelihood geometry of the model, for any finite group of isometries $G$. At the population level, we characterize the set of $\theta$'s for which the Fisher information is invertible and, in general, we give a full description of the null space of the Fisher information in terms of $\theta^*$ and its interaction with $G$ and we study higher order derivatives of the population log-likelihood function. As a consequence, in Section \ref{Sec:Statistics}, we derive statistical properties of the MLE and we characterize the pointwise rates of convergence of this estimator, when projected on orthogonal subspaces. In brief, we show that $\theta^*$ can always be decomposed into two components: One for which the MLE achieves the parametric rate $n^{-1/2}$ and one for which it achieves the slower rate $n^{-1/4}$, and we give a precise description of this decomposition. As a byproduct, we show that the pointwise estimation rate of MLE is never worse than $n^{-1/4}$. Finally, in Section \ref{Sec:Examples}, we illustrate our results by considering some examples of groups of isometries. Some of the proofs and intermediate lemmas are deferred to the appendix.

\paragraph{Notation}
In this work, the ambient dimension is denoted by $d$. The Euclidean norm in $\R^d$ is denoted by $\|\cdot\|$ and the transpose of a vector $u\in\R^d$ is $u^\top$. 

The complement of a set or an event $A$ is denoted by $A^{\complement}$ and the cardinality of a set $A$ is denoted by $|A|$.

If $f:\R^d\to\R$ is a smooth function, we denote by $\diff^k f(x)$ its $k$-th differential at a point $x\in\R^d$: It is a symmetric function of $k$ $d$-dimensional variables. When $f:\R^d\times\R^d\to\R$ is a function of two variables $y$ and $\theta$ that is smooth with respect to $\theta$, we denote by $\partial_\theta^k f(y,\theta)$ its $k$-th differential with respect to $\theta$ at the point $(y,\theta)$: This is also a symmetric function of $k$ $d$-dimensional variables. When $k=1$ (resp. $k=2$), we also write $\DS \partial_\theta(y,\theta)(u)=u^\top\frac{\partial f}{\partial\theta}(y,\theta)$ (resp. $\DS \partial_\theta^2(y,\theta)(u,v)=u^\top\frac{\partial^2 f}{\partial\theta\partial\theta^\top}(y,\theta)v$).

We let $G$ be a subgroup of isometries, which we suppose fixed and known. For $\theta\in\R^d$, we denote by $\DS \PP_\theta=\frac{1}{|G|}\sum_{g\in G}\mathcal N(g\theta,I)$, where $\mathcal N$ is the symbol for Gaussian distributions and $I$ is the identity matrix in $\R^{d\times d}$. The corresponding expectation, variance and covariance operators are denoted by $\E_{\theta}$, $\Var_\theta$ and $\cov_\theta$, respectively.

\section{Likelihood geometry of the model} \label{SectionLikelihood}

Let $\theta^*\in\R^d$ be fixed and consider a sequence $Y,Y_1,Y_2,\ldots$ of i.i.d. random vectors distributed according to $\PP_{\theta^*}$. The corresponding log-likelihood is defined, for all positive integer $n$, as
\begin{equation} \label{loglikelihood}
	\hat\Psi_n(\theta)=\frac{1}{n}\sum_{i=1}^n \log L(Y_i,\theta), \quad \forall \theta\in\R^d,
\end{equation}
where $L(y,\theta)$ is the density of $\PP_{\theta}$ with respect to the Lebesgue measure on $\R^d$:
\begin{equation}
	L(y,\theta)=\frac{1}{|G|(2\pi)^{d/2}}\sum_{g\in G}e^{-\frac{1}{2}\|y-g\theta\|^2}, \quad \forall y,\theta\in\R^d.
\end{equation}
In general, the factor $1/n$ does not appear in the definition of the log-likelihood, but we include it so the expectation of $\hat\Psi(\theta)$ is the population log-likelihood of the parametric model, given by 
\begin{equation}
	\Psi(\theta)=\E_{\theta^*}[\log L(Y,\theta)].
\end{equation}

Then, the Fisher information of the model is defined as $I(\theta^*)=-\diff^2\Psi(\theta^*)$. 

Denote by $H$ the stabilizer of $\theta^*$, i.e., the collection of all elements $g\in G$ such that $g\theta^*=\theta^*$, and by $\DS \bar H= \frac{1}{|H|}\sum_{g\in H}g$. Note that $H$ is always nonempty since at least the identity belongs to $H$, hence, $\bar H$ is always well defined. Moreover, it is easy to check that $H$ is a subgroup of $G$.

\begin{theorem} \label{maintheorem}
	The null space of $I(\theta^*)$ coincides with the null space of $\bar H$, i.e.,  
	\begin{equation}
		\forall u\in\R^d, u^\top I(\theta^*)u=0\iff \bar H u=0.
	\end{equation}
	Moreover, if $u$ is in the nullspace of $I(\theta^*)$, then $\diff^4\Psi(\theta^*)(u,u,u,u)=0$ only if $u=0$.
\end{theorem}

The proof of Theorem \ref{maintheorem} relies on the following result, the first two conclusions of which are folklore in parametric statistics. The third conclusion of the next lemma is essential in our analysis, since it drives the statistical rates that we discuss in Section \ref{Sec:Statistics}. This lemma is quite technical, hence, we defer its proof to the appendix. However, it is easy to check that the assumptions are all satisfied in our Gaussian mixture model.

\begin{lemma} \label{MainLemma}
	Let $(\QQ_\theta)_{\theta\in\Theta}$ be a family of probability distributions on some abstract space $\mathcal Y$, where $\Theta\subseteq\R^d$ and let $\theta^*$ be in the interior of $\Theta$. Let $\E_{\theta^*}$ and $\Var_{\theta^*}$ stand for the expectation and the variance operators associated with $\QQ_{\theta^*}$, respectively. Assume that there exists a measure $\mu$ on $\mathcal Y$ and a neighborhood $\mathcal V$ of $\theta^*$ in $\Theta$ such that the following holds:
\begin{itemize}
	\item $\QQ_\theta$ is absolutely continuous with respect to $\mu$ for all $\theta\in\mathcal V$;
	\item The support of $\QQ_{\theta}$ does not depend on $\theta$;
	\item The density $\DS L(y,\theta)=\frac{\diff \QQ_\theta}{\diff \mu}(y), y\in\mathcal Y, \theta\in\Theta$, is five times differentiable with respect to $\theta\in\mathcal V$, for $\mu$-almost all $y\in\mathcal Y$;
	\item For $\mu$-almost all $y\in\mathcal Y$, the first four derivatives of $L(y,\cdot)$ with respect to $\theta$ are uniformly bounded on $\mathcal V$ by $\mu$-integrable functions and the first four derivatives of $\log L(y,\theta)$ with respect to $\theta$ are uniformly bounded on $\mathcal V$ by $\QQ_{\theta^*}$-integrable functions. 
\end{itemize}
Denote by $\Psi(\theta)=\E_{\theta^*}\crc{\log L(Y,\theta)}$, for all $\theta\in\mathcal V$. Then, 
\begin{enumerate}
	\item[(i)] $\DS \frac{\partial\Psi}{\partial\theta}(\theta^*)=\E_{\theta^*}\crc{\frac{\partial \log L}{\partial \theta}(Y,\theta^*)}=0$;
	\item[(ii)] For all $u\in\R^d$, $\DS \diff^2\Psi(\theta^*)(u,u)=-\Var_{\theta^*}\crc{u^\top \frac{\partial \log L}{\partial \theta}(Y,\theta^*)}$;
	\item[(iii)] For all $w\in\R^d$ such that $\diff^2\Psi(\theta^*)(w,w)=0$, it holds that $\diff^3\Psi(\theta^*)(w,w,w)=0$ and that $\diff^4\Psi(\theta^*)(w,w,w,w)=-3\Var_{\theta^*}\crc{\frac{1}{L(Y,\theta^*)}w^\top\frac{\partial^2 L}{\partial\theta\partial\theta^\top}(Y,\theta^*)w}$.
\end{enumerate}
\end{lemma}

Before giving the proof of Theorem \ref{maintheorem}, we state one more lemma, which gives a better description of the operator $\bar H$ that characterizes the nullspace of the Fisher information. We let $\sim$ be the equivalence relation on $G$ defined by $g\sim g'\iff g\theta^*=g'\theta^*\iff g^{-1}g'\in H$ and we denote by $E=G/H$ the set of equivalence classes. In other words, $E$ partitions $G$ into subsets such that any two elements $g,g'\in G$ are in the same set $S\in E$ if and only if $g\theta^*=g'\theta^*$. For example, $H\in E$ and any $S\in E$ is of the form $S=gH=\{gh: h\in H\}$ for some $g\in G$ (or, more precisely, for any $g\in S$). As a consequence, all the sets in $E$ have the same cardinality: $|S|=|H|$ for all $S\in E$. For all $S\in E$, let $\bar S=\frac{1}{|S|}\sum_{g\in S}g$.

\begin{lemma} \label{LemmaProj}
	\begin{enumerate}
		\item[(i)] The map $\bar H$ is the orthogonal projection onto the set of all vectors $u\in\R^d$ that are stabilized by $H$, i.e., $\{u\in\R^d: hu=u, \forall h\in H\}$.
		\item[(ii)] For all $S\in E$ and $g\in S$, $\bar S=g\bar H$.
		\item[(iii)] Let $v,w\in\R^d$ such that $\bar Hv=v$ and $\bar Hw=0$. Then, for all $S\in E$ and any $g\in S$, $\bar Sv=gv$ and $\bar Sw=0$. 
	\end{enumerate}
\end{lemma}

The proof of Lemma \ref{LemmaProj} is deferred to the appendix. 
We are now in a right position to give the proof of Theorem \ref{maintheorem}.

\begin{proof}[Proof of Theorem \ref{maintheorem}]
First, note that the assumptions of Lemma \ref{MainLemma} are easily verified for the family $(\PP_\theta)_{\theta\in\R^d}$. Therefore, for all $u\in\R^d$, 
\begin{equation} \label{mainthm1}
\E_{\theta^*}\crc{u^\top \frac{\partial \log L}{\partial \theta}(Y,\theta^*)}=0.
\end{equation}
Now, $u$ is in the null space of $I(\theta^*)$ if and only if $u^\top I(\theta^*)u=0$, since $I(\theta^*)$ is positive semidefinite, i.e., if and only if $\Var_{\theta^*}\crc{u^\top \frac{\partial \log L}{\partial\theta}(Y,\theta^*)}=0$, again by Lemma \ref{MainLemma}. Hence, the random variable $\DS u^\top \frac{\partial \log L}{\partial\theta}(Y,\theta^*)$ must be constant $\PP_{\theta^*}$-almost surely and by \eqref{mainthm1}, it must be zero. In other words, $\DS u^\top \frac{\partial \log L}{\partial \theta}(y,\theta^*)=0$, for all $y\in\R^d$. A straightforward computation shows that this is equivalent to
\begin{equation*} 
	\sum_{g\in G} e^{-\frac{1}{2}\|y-g\theta^*\|^2}(y-g\theta^*)^\top gu =0, 
\end{equation*}
for all $y\in\R^d$ which, by Lemma \ref{LemmaProj}, can be rewritten as
\begin{equation} \label{mainthm3}
	\sum_{S\in E} e^{-\frac{1}{2}\|y-\bar S\theta^*\|^2}(y-\bar S\theta^*)^\top \bar Su =0.
\end{equation}
As a straightforward consequence of Lemma \ref{LemmaProj}, for all $S\in E$, $\|\bar S\theta^*\|^2=\|\theta^*\|^2$ and \eqref{mainthm3} becomes
\begin{equation} \label{mainthm4}
	\sum_{S\in E} e^{y^\top\bar S\theta^*}(y-\bar S\theta^*)^\top \bar Su =0.
\end{equation}
In particular, taking $y=0$ yields that $\DS \sum_{S\in E}(\bar S\theta^*)^\top\bar Su=0$. For all $S\in E$, write $\bar S$ as $g\bar H$ for (any) $g\in S$, as we have seen above; Then, $\bar S^\top\bar S=\bar H^\top g^\top g\bar H=\bar H^\top\bar H=\bar H$, yielding that $(\theta^*)^\top u=0$ and \eqref{mainthm4} becomes
\begin{equation} \label{mainthm5}
	\sum_{S\in E} e^{y^\top\bar S\theta^*}y^\top \bar Su =0.
\end{equation}
From Lemma \ref{LemmaProj}, it is clear that the vectors $\bar S\theta^*$ are pairwise distinct. Now, fix $S_0\in E$ and let $C_0=\left\{y\in\R^d: y^\top\bar {S_0}\theta^*>y^\top \bar S\theta^*, \forall S\in E, S\neq S_0\right\}$: This is an open, nonempty set. Let us show that for all $y\in C_0$, $y^\top \bar S_0 u=0$. This will yield that $\bar S_0 u$ is necessarily in the orthogonal of $C_0$, which is $\{0\}$ since $C_0$ is open. Let $y\in C_0$. Then, \eqref{mainthm5} implies that
\begin{equation}
	0=\lim_{t\to\infty}e^{ty^\top\bar S_0\theta^*}\sum_{S\in E} e^{ty^\top\bar S\theta^*}y^\top \bar Su =y^\top \bar S_0u,
\end{equation}
which is what we wanted to prove.

Conversely, if $\bar Hu=0$, then $\bar Su=0$ for all $S\in E$, it is true that \eqref{mainthm3} must hold. Therefore, reverse-engineering the previous computations yields that $\DS u^\top \frac{\partial \log L}{\partial\theta}(Y,\theta^*)=0$ $\PP_{\theta^*}$-almost surely, yielding that $\diff^2\Psi(\theta^*)(u,u)=0$, which concludes the proof of the first part of the theorem.

Now, let $u\in\R^d$ such that $\bar Hu=0$. Then, by Lemma \ref{MainLemma}, $\diff^4\Psi(\theta^*)(u,u,u,u)=0$ if and only if the random variable $\DS \frac{1}{L(Y,\theta^*)}u^\top\frac{\partial^2 L}{\partial\theta\partial\theta^\top}(Y,\theta^*)u$ is constant $\PP_{\theta^*}$-almost surely. Since its expectation with respect to $\PP_{\theta^*}$ is zero, then it must be equal to zero $\PP_{\theta^*}$-almost surely. In other words, $\DS u^\top\frac{\partial^2 L}{\partial\theta\partial\theta^\top}(y,\theta^*)u=0$, for all $y\in\R^d$. Up to some constant factor $C>0$,
\begin{align*}
	u^\top\frac{\partial^2 L}{\partial\theta\partial\theta^\top}(y,\theta^*)u & = C\sum_{g\in G}e^{-\frac{1}{2}\|y-g\theta^*\|^2}u^\top\prt{g^\top (y-g\theta^*)(y-g\theta^*)^\top g-gg^\top}u \nonumber \\
	& = C\sum_{S\in E}\sum_{g\in S}e^{-\frac{1}{2}\|y-\bar S\theta^*\|^2}u^\top\prt{g^\top (y-\bar S\theta^*)(y-\bar S\theta^*)^\top g-I}u \nonumber \\
	& = C\sum_{S\in E}e^{-\frac{1}{2}\|y-\bar S\theta^*\|^2}\sum_{g\in S}\prt{u^\top g^\top (y-\bar S\theta^*)(y-\bar S\theta^*)^\top gu-\|u\|^2},
\end{align*} 
hence, for all $y\in\R^d$, it must hold that
\begin{equation} \label{mainthm11}
	\sum_{S\in E}e^{-\frac{1}{2}\|y-\bar S\theta^*\|^2}\sum_{g\in S}\prt{u^\top g^\top (y-\bar S\theta^*)(y-\bar S\theta^*)^\top gu-\|u\|^2}=0.
\end{equation}
Note that for all $S\in E$ and $g\in S$, $\bar S^\top gu=\bar H^\top g^\top gu=\bar H^\top u=\bar Hu=0$, where we used the facts that $\bar S=g\bar H$ and that $\bar H$ is symmetric, by Lemma \ref{LemmaProj}. Thus, also noting that for all $S\in E$, $\|\bar S\theta^*\|^2=\|\theta^*\|^2$, \eqref{mainthm11} yields
\begin{equation*} 
	\sum_{S\in E}e^{y^\top\bar S\theta^*}\sum_{g\in S}\prt{u^\top g^\top yy^\top gu-\|u\|^2}=0, 
\end{equation*}
for all $y\in\R^d$. In particular, for $y=0$, this directly yields that $\|u\|^2=0$, i.e., $u=0$.

\end{proof}

As a second important consequence of Lemma \ref{LemmaProj}, the following corollary holds.
\begin{corollary} \label{CorollaryDefinite}
	The Fisher information $I(\theta^*)$ is definite if and only if all the modes $g\theta^*, g\in G$, of $\PP_{\theta^*}$ are pairwise distinct.
\end{corollary}

\begin{proof}
	By Theorem \ref{maintheorem} and the first part of Lemma \ref{LemmaProj}, $I(\theta^*)$ is definite $\iff$ the projection $\bar H$ is invertible $\iff$ its rank is equal $d$ $\iff$ its trace is equal to $d$. Since the trace of any isometry is at most $d$, $\Tr(\bar H)=\frac{1}{|H|}\sum_{g\in H}\Tr(g)\leq d$ with equality if and only if $\Tr(g)=d$ for all $g\in H$ $\iff$ $g=I$ for all $g\in H$ $\iff$ $H=\{I\}$, i.e., $g\theta^*\neq \theta^*$, for all $g\in G\setminus\{I\}$, i.e., the modes of $\Phi$ are pairwise distinct.
\end{proof}

Geometrically, assuming that all the $g\theta^*, g\in G$ are pairwise distinct can be interpreted as assuming that $\theta^*$ exhibits no symmetries or rotational invariances that are encoded in $G$: For instance, if $G$ contains a reflexion around some subspace, saying that the centers of $\PP_{\theta^*}$ are pairwise distinct implies that $\theta^*$ can not be symmetric with respect to that subspace.

\begin{remark}
	The operator $\bar H$ has the following geometric interpretation, in the likelihood landscape of the model. For $\theta\in\R^d$, let $\Theta(\theta)=\{g\theta:g\in G\}$ be the identified set associated with $\theta$, $\textsf{deg}(\theta)=\left|\Theta(\theta)\right|$ its cardinality, which we call \textit{degree of identifiability} of $\theta$ and $H(\theta)=\{g\in G:g\theta=\theta\}$. It is easy to see that $\textsf{deg}(\theta)=|G|/|H(\theta)|$: The degree of identifiability of $\theta$ is always a divider of $|G|$. For instance, $\textsf{deg}(\theta)=1$ means that 
$\theta$ is uniquely identified, in the sense that for all $\theta'\in\R^d$, $\PP_{\theta'}=\PP_{\theta}\Rightarrow \theta'=\theta$, and the larger $\textsf{deg}(\theta)$ is, the less $\theta$ is identifiable in the model.	Now, let $\mathcal U=\{u\in\R^d:\textsf{deg}(\theta^*+tu)>\textsf{deg}(\theta^*) \mbox{ when } |t| \mbox{ is small enough}\}$: This is the set of directions in which a small perturbation of $\theta^*$ increases the degree of identifiability. Now, note that for all $u\in\R^d$ and $t\in\R$ with small enough $|t|$, $H(\theta^*+tu)$ is a subgroup of $H(\theta^*)$. If $u=0$, this is trivial. If $u\neq 0$, let $|t|<(2\|u\|)^{-1}\min_{g\in G\setminus H(\theta^*)}\|g\theta^*-\theta^*\|$; Then, for all $g\in G\setminus H(\theta^*)$, $\|g(\theta^*+tu)-(\theta^*+tu)\|\geq \|g\theta^*-\theta^*\|-|t|\|gu-u\|\geq \|g\theta^*-\theta^*\|-2|t|\|u\|>0$, implying $g\notin H(\theta^*+tu)$. Therefore, $\mathcal U$ is the set of directions $u\in\R^d$ such that if $|t|$ is small enough, $H(\theta^*+tu)$ is a strict subgroup of $H(\theta^*)$. Geometrically, this means that $\mathcal U$ is the set of directions $u\in\R^d$ that pushes away the colliding modes of the log-likelihood: If we denote by $\Psi_{\theta}(\cdot)=\E_{\theta}[\log L(Y,\cdot)]$, then some of the colliding modes of $\Psi_{\theta^*}(\cdot)$ (i.e., the $h\theta^*, h\in H$) become distinct modes for $\Psi_{\theta^*+tu}(\cdot)$, for small enough $|t|$, while no other modes merge. Now, we can rewrite $\mathcal U=\{u\in\R^d:\mbox{ if } |t| \mbox{ is small enough}, \exists h\in H, h(\theta^*+tu)\neq \theta^*+tu\}=\{u\in\R^d: \exists h\in H, hu\neq u\}=\R^d\setminus \textsf{im}(\bar H)$: This is the complement of the range of $\bar H$. 
\end{remark}


Applied to Gaussian mixtures with group invariance, Lemma \ref{MainLemma} also yields the following important corollary.

\begin{corollary} \label{CorollaryTaylor}
	Let $\theta\in\R^d$ and $g_0\in G$ such that $\|g_0 \theta-\theta^*\|=\min_{g\in G}\|g\theta-\theta^*\|$. Write $g_0\theta-\theta^*=v+w$, where $v,w\in\R^d$ satisfy $\bar Hv=v$ and $\bar Hw=0$. Then, there exists a positive constant $C$ such that if $\|v\|$ and $\|w\|$ are small enough,
	\begin{equation*}
		\Psi(\theta)-\Psi(\theta^*) \leq -C\prt{\|v\|^2+\|w\|^4}.
	\end{equation*}
\end{corollary}

Note that the vectors $v$ and $w$ in Corollary \ref{CorollaryTaylor} are uniquely defined: $v=\bar H(g_0\theta-\theta^*)$ and $w=(I-\bar H)(g_0\theta-\theta^*)$.

\begin{proof}
	A Taylor expansion yields:
\begin{align}
	\Psi(\theta)-\Psi(\theta^*) & = \diff\Psi(\theta^*)(u) +\frac{1}{2}\diff^2\Psi(\theta^*)(u,u)+\frac{1}{6}\diff^3\Psi(\theta^*)(u,u,u) \nonumber \\
	& \quad \quad \quad \quad +\frac{1}{24}\diff^4\Psi(\theta^*)(u,u,u,u)+o(\|u\|^4) \nonumber \\
	& =: \textsf{I}+\textsf{II}+\textsf{III}+\textsf{IV}+o(\|v\|^2+\|w\|^4). \label{BigExp1}
\end{align}
Since $\theta^*$ is a maximum of $\Psi$ and $\Psi$ is differentiable, $\textsf{I}=0$ in \eqref{BigExp1}. For the second term, one has
\begin{align*}
	\diff^2\Psi(\theta^*)(u,u) & = u^\top \frac{\partial^2 \Psi}{\partial\theta\partial\theta^\top}(\theta^*)u \\
	& = v^\top \frac{\partial^2 \Psi}{\partial\theta\partial\theta^\top}(\theta^*)v+2v^\top \frac{\partial^2 \Psi}{\partial\theta\partial\theta^\top}(\theta^*)w+w^\top \frac{\partial^2 \Psi}{\partial\theta\partial\theta^\top}(\theta^*)w.
\end{align*}
Since $\bar Hw=0$, $w$ is in the null space of the negative semidefinite matrix $\DS \frac{\partial^2 \Psi}{\partial\theta\partial\theta^\top}(\theta^*)$, by Theorem \ref{maintheorem}, yielding $\DS \textsf{II}=v^\top \frac{\partial^2 \Psi}{\partial\theta\partial\theta^\top}(\theta^*)v=-\Var_{\theta^*}\crc{v^\top\frac{\partial\log L}{\partial\theta}(Y,\theta^*)}$. For the third term in \eqref{BigExp1}, one has
\begin{align*}
	\diff^3 \Psi(\theta^*)(u,u,u) & = \diff^3\Psi(\theta^*)(v,v,v)+3\diff^3\Psi(\theta^*)(v,v,w) \\
	& \quad +3\diff^3\Psi(\theta^*)(v,w,w)+\diff^3\Psi(\theta^*)(w,w,w).
\end{align*}
In the right hand side of the last display, the last term is zero, by Theorem \ref{maintheorem}. Moreover, the first two terms are $o(\|v\|^2)$, hence, in \eqref{BigExp1}, $\DS \textsf{III} = 3\diff^3\Psi(\theta^*)(v,w,w)+o(\|v\|^2)$. Finally, for the fourth term in \eqref{BigExp1}, write, in the same fashion as for the other terms,
\begin{align*}
	\diff^4\Psi(\theta^*)(u,u,u,u) & = 4\diff^4\Psi(\theta^*)(v,w,w,w)+\diff^4\Psi(\theta^*)(w,w,w,w)+o(\|v\|^2) \\
	& = 4\diff^4\Psi(\theta^*)(v,w,w,w)-3\Var_{\theta^*}\crc{w^\top\frac{\partial^2\log L}{\partial\theta\partial\theta^\top}(Y,\theta^*)w} \\
	& \quad \quad +o(\|v\|^2),
\end{align*}
where we used Theorem \ref{maintheorem} for the last equality.

Wrapping up, one obtains from \eqref{BigExp1} and the intermediate computations, 
\begin{align}
	\Psi(\theta)-\Psi(\theta^*) = & -\frac{1}{2}\Var_{\theta^*}\crc{v^\top\frac{\partial \log L}{\partial\theta}(Y,\theta^*)}+\frac{1}{2}\diff^3\Psi(\theta^*)(v,w,w) \nonumber \\
	& -\frac{1}{8}\Var_{\theta^*}\crc{w^\top\frac{\partial^2\log L}{\partial\theta\partial\theta^\top}(Y,\theta^*)w}+\frac{1}{6}\diff^4\Psi(\theta^*)(v,w,w,w) \\
	& +o(\|v\|^2+\|w\|^4). \label{BigExp2}
\end{align}

Now, we make use of the following result:

\begin{lemma}\label{higherorderid}
\begin{equation*}
	\diff^3\Psi(\theta^*)(v,w,w)=-\cov_{\theta^*}\prt{v^\top\frac{\partial\log L}{\partial\theta}(Y,\theta^*),w^\top\frac{\partial^2\log L}{\partial\theta\partial\theta^\top}(Y,\theta^*)w}
\end{equation*}
and
\begin{equation*}
	\diff^4\Psi(\theta^*)(v,w,w,w)=0.
\end{equation*}
\end{lemma}

Thus, \eqref{BigExp2} implies that 
\begin{align*}
	\Psi(\theta)-\Psi(\theta^*) = & -\frac{1}{8}\Var_{\theta^*}\crc{2v^\top\frac{\partial\log L}{\partial\theta}(Y,\theta^*)+w^\top\frac{\partial^2\log L}{\partial\theta\partial\theta^\top}(Y,\theta^*)w} \\
	& +o(\|v\|^2+\|w\|^4).
\end{align*}

Finally, the following lemma, which we prove in the appendix, allows to conclude this proof.

\begin{lemma} \label{IntermediateLemma4}
	There exists a constant $C>0$ that does not depend on $v$ and $w$ such that 
	\begin{equation*}
		\Var_{\theta^*}\crc{2v^\top\frac{\partial\log L}{\partial\theta}(Y,\theta^*)+w^\top\frac{\partial^2\log L}{\partial\theta\partial\theta^\top}(Y,\theta^*)w} \geq C\prt{\|v\|^2+\|w\|^4}.
	\end{equation*}
\end{lemma}

\end{proof}

\section{Statistical rates} \label{Sec:Statistics}

Now that we have understood the likelihood geometry of the statistical model, we are in a position to state some statistical results about the MLE $\hat \theta_n$. Recall that the MLE maximizes $\hat\Psi_n(\theta)$, which was defined in \eqref{loglikelihood}. As we have already explained in the introduction, we measure the performance of $\hat\theta_n$ by $\DS \rho(\hat\theta_n,\theta^*)=\min_{g\in G}\|g\hat\theta_n-\theta^*\|$.
The first result is that $\hat\theta_n$ is consistent.

\begin{theorem} \label{thm:consistency}
	For all $\theta^*\in\R^d$, $\rho(\hat\theta_n,\theta^*) \underset{n\to\infty}{\longrightarrow} 0$ in $\PP_{\theta^*}$-probability.
\end{theorem}

\begin{proof}

Let $C=\|\theta^*\|^2+d+1$ and consider the event $\mathcal A$ when $\DS \frac{1}{n}\sum_{i=1}^n \|Y_i\|^2\leq C$. Then, since $\|Y_1\|^2,\ldots,\|Y_n\|^2$ are i.i.d subexponential random variables with mean $C-1$, $\PP_{\theta^*}\crc{\mathcal A}\rightarrow 1$, as $n\to\infty$. Let the event $\mathcal A$ hold. Then, for all $\theta\in\R^d$ with $\|\theta\|>\sqrt{3C}$,
\begin{align*}
	\hat\Psi_n(\theta) & = \frac{1}{n}\sum_{i=1}^n \log\prt{\frac{1}{(2\pi)^{d/2}|G|}\sum_{g\in G}e^{-\frac{1}{2}\|Y_i-g\theta\|^2}} \\
	& = -\frac{d}{2}\log(2\pi)+\frac{1}{n}\sum_{i=1}^n\log\prt{\frac{1}{|G|}\sum_{g\in G}e^{-\frac{1}{2}\|Y_i-g\theta\|^2}} \\
	& \leq -\frac{d}{2}\log(2\pi)+\frac{1}{n}\sum_{i=1}^n\log\prt{e^{-\frac{3C}{2}+\|Y_i\|^2}} = -\frac{d}{2}\log(2\pi)-\frac{3C}{2}+\frac{1}{n}\sum_{i=1}^n\|Y_i\|^2,
\end{align*}
where we used, in the first inequality, that $\|g\theta\|=\|\theta\|$ for all $g\in G$ and $\|Y_i-g\theta\|^2\geq \frac{3C}{2}-\|Y_i\|^2$.
Now, note that $\DS \hat\Psi_n(0)=-\frac{d}{2}\log(2\pi)-\frac{1}{2n}\sum_{i=1}^n \|Y_i\|^2$. Hence, if $\mathcal A$ holds, then $\hat \Psi_n(\theta)<\hat \Psi_n(0)$ for all $\theta\in\R^d$ with $\|\theta\|>\sqrt{3C}$. Thus, if $\mathcal A$ holds, it must be true that $\|\hat\theta_n\|\leq \sqrt{3C}$. Hence, for all $\varepsilon>0$,
\begin{equation*}
	\PP_{\theta^*}\crc{\rho(\hat\theta,\theta^*)>\varepsilon} \leq \PP_{\theta^*}\crc{\rho(\hat\theta,\theta^*)>\varepsilon, \|\hat\theta_n\|\leq \sqrt{3C}}+\PP_{\theta^*}\crc{\mathcal A^{\complement}}.
\end{equation*}
The first term goes to zero as $n\to\infty$ by \cite[Theorem 5.14]{vaart_1998} and we have already established that the second term goes to zero, which ends the proof of the lemma.

\end{proof}

Hence, when the Fisher information is invertible, \cite[Theorem 5.41]{vaart_1998} yields the asymptotic normality of $\hat\theta_n$, in the following sense.

\begin{corollary}
	Let $\theta^*\in\R^d$ be such that for all $g\in G$ with $g\neq I$, $g\theta^*\neq\theta^*$. Then, there exists a (random) sequence $g_n\in G$ such that 
	\begin{equation*}
		\sqrt{n}(\hat\theta_n-g_n\theta^*)\underset{n\to\infty}{\longrightarrow} \mathcal N\prt{0,I(\theta^*)},
	\end{equation*}
where $I(\theta^*)\in\R^{d\times d}$ is the Fisher information matrix.
\end{corollary}

In particular, it is possible to get asymptotic confidence regions around $\hat\theta_n$ that contain a version of $\theta^*$ (i.e., some $g\theta^*$) with high probability asymptotically.

By Theorem \ref{maintheorem}, the Fisher information is not always invertible. However, Theorem \ref{maintheorem} suggests that in any case, the rate for estimation of $\theta^*$ should not be too slow, since the log-likelihood always have some nonzero curvature, at least at the fourth order. Moreover, since $\Psi$ has some positive second order curvature at least along some directions, these directions should define subspaces along which $\theta^*$ can be estimated at the standard rate $n^{-1/2}$. The following version of \cite[Lemma 14.4]{Kosorok2008} allows to get different rates for the components of one and the same $\theta^*$. By components of $\theta^*$ we mean the orthogonal projections of $\theta^*$ onto linear subspaces that are given by the null space of the Fisher information $I(\theta^*)$ and its orthogonal. We denote by $\E^*$ the outer expectation. 

\begin{lemma} \label{lemmaKosorok}
	Let $(M_n)_n\geq 1$ be a sequence of real valued stochastic processes defined on a space $\Theta$ and let $M:\Theta\to\R$ be a given function. Let $\Theta$ be equipped with a semimetric $\rho$, i.e., a function that satisfies all the axioms of a distance but the definiteness. Let $\theta^*\in\Theta$ and $\rho_{\theta^*}$ be a nonnegative function defined on $\Theta$ such that $\rho_{\theta^*}(\theta^*)\rightarrow 0$ as $\rho(\theta,\theta^*)\rightarrow 0$.
Assume that there exist positive constants $c_1, c_2$ and $\delta_0>0$ such that the following holds:
\begin{itemize}
	\item $M(\theta)-M(\theta^*)\leq -c_1\rho_{\theta^*}(\theta^*)^2$, for all $\theta\in\Theta$ with $\rho(\theta,\theta^*)\leq\delta_0$;
	\item $\DS \E^*\crc{\sup_{\rho_{\theta^*}(\theta^*)\leq\delta} \sqrt n |(M_n-M)(\theta)-(M_n-M)(\theta^*)|} \leq c_2\delta$, for all $\delta\in (0,\delta_0)$ and $n$ large enough.
\end{itemize}
If $\hat\theta_n\in\Theta$ satisfies $M_n(\hat\theta_n)=\sup_{\theta\in\Theta}M_n(\theta)$ for all $n$ large enough, and if $\hat\theta_n$ converges to $\theta^*$ in outer probability, then
\begin{equation*}
	\sqrt n\rho_{\theta^*}(\hat\theta_n)=O_\PP(1).
\end{equation*}	
\end{lemma}

Unlike standards results (see, e.g., \cite[Chapter 5]{vaart_1998}, Lemma \ref{lemmaKosorok} allows to get different rates simultaneously for one and the same unknown vector $\theta^*$. 
In our setting, $M_n(\theta)=\hat\Psi_n(\theta)$ and $M(\theta)=\Psi(\theta)$ and we let $\DS \rho(\theta,\theta')=\min_{g\in G}\|g\theta_1-\theta_2\|$ and $\rho_{\theta^*}(\theta)^2=\|\bar H(g_0\theta-\theta^*)\|^2+\|(I-\bar H)(g_0\theta-\theta^*)\|^4$ where $g_0\in G$ is the minimizer of $\|g\theta-\theta^*\|$ for $g\in G$. The consequence of this lemma, in our setting, is as follows. The unknown vector $\theta^*$ has two components, one in the range of $\bar H$ and one in its orthogonal. The former is estimated at the usual parametric speed $n^{-1/2}$, whereas the latter is estimated at the slower, but not too slow, speed $n^{-1/4}$. This is made formal in the following theorem.

\begin{theorem} \label{ThmStats}
	Let $\theta^*\in\R^d$ and let $Y_1,Y_2,\ldots$ be a sequence of i.i.d. random variables distributed according to $\PP_{\theta^*}$. For $n\geq 1$, let $\hat\theta_n$ be the MLE of $\theta^*$ obtained from $Y_1,\ldots,Y_n$. Then, there exists a sequence $(g_n)_{n\geq 1}$ of elements of $G$, such that
\begin{enumerate}
	\item[(i)] $\DS n^{1/2}\|\bar H(g_n\hat\theta_n-\theta^*)\| = O_{\PP_{\theta^*}}(1)$;
	\item[(ii)]	$\DS n^{1/4}\|(I-\bar H)(g_n\hat\theta_n-\theta^*)\| =  O_{\PP_{\theta^*}}(1)$.
\end{enumerate}
\end{theorem}

\begin{proof}

In order to prove this theorem, we check that the assumptions from Lemma \ref{lemmaKosorok} are satisfied. 

The first assumption is proven in Corollary \ref{CorollaryTaylor} and the consistency of $\hat\theta_n$ with respect to $\rho$ is ensured by Theorem \ref{thm:consistency}

Finally, let $\delta>0$ and $\mathcal F_{\delta}=\{\log L(\cdot,\theta)-\log L(\cdot,\theta^*):\theta\in\R^d, \rho_{\theta^*}(\theta)\leq \delta\}$. Recall the definition of $\rho_{\theta^*}(\theta)$: Let $g_0\in G$ such that $\DS \|g_0\theta-\theta^*\|=\min_{g\in G}  \|g\theta-\theta^*\|$ and write $g_0\theta-\theta^*=v+w$, where $v,w\in\R^d$ with $\bar Hv=v$ and $\bar Hw=0$. Then, $\rho_{\theta^*}(\theta)^2=\|v\|^2+\|w\|^4$. The next lemma is proved in the appendix.

\begin{lemma} \label{TaylorExp}
	There exists $F\in L^2(\PP_{\theta^*})$ such that, for $\delta>0$ small enough and for all $f\in \mathcal F_\delta$, $|f|\leq \delta F$. 
\end{lemma}

Therefore, the second assumption of Lemma \ref{lemmaKosorok} is satisfied, thanks to \cite[Lemma 19.38]{vaart_1998}. We have proven that all the assumptions of Lemma \ref{lemmaKosorok} are satisfied, which yields Theorem \ref{ThmStats}.

\end{proof}

Perhaps surprinsingly, these rates do not depend on the size of $H$ and this theorem shows that any $\theta^*$ could be estimated, up to identifiability, at least at the rate $n^{-1/4}$ via the MLE. However, Theorem \ref{ThmStats} does not provide uniform bounds. In fact, the $O_{\PP_{\theta^*}}$ signs contain constants that depend on $\theta^*$ and may become arbitrarily large. For instance, one expects that the constants hidden in the $n^{1/2}$ should blow up when two modes $\theta^*$ and $g\theta^*$, for $g\notin H$, are distinct but arbitrarily close to each other.

\section{Examples} \label{Sec:Examples}

Here, we review for important examples, not to mention the trivial case when $G=\{I\}$, where the previous results yield, as expected, the definiteness of the Fisher information, no matter the value of $\theta^*$. 

\subsection{If $G=\{-I,I\}$}

In this case, $H$ can be either the trivial subgroup, when $\theta^*\neq 0$ or $G$ itself, when $\theta^*=0$. In the first case, $\bar H=I$, hence $\theta^*$ is estimated at the parametric speed $n^{-1/2}$ by the MLE. In the second case, $\bar H=0$ and all components of $\theta^*$ are estimated at the slower rate $n^{-1/4}$.

Here, $\theta^*$ is to be recovered up to a global sign flip. \newline

\subsection{If $G=\left\{\textsf{Diag}(\omega_1,\ldots,\omega_d):\omega_1,\ldots,\omega_d\in\{-1,1\}\right\}$}

Here, $G$ is the isometry subgroup spanned by all the reflexions with respect to the hyperplanes of the form $\{(u_1,\ldots,u_d)^\top\in\R^d: u_j=0\}, j=1,\ldots,d$.
Let $B=\{j=1,\ldots,d:\theta_j^*=0\}$ and let $p=|B|$. Then, $H=\{\textsf{Diag}(\omega):\omega\in\{-1,1\}^d, \omega_i=1, \forall i\notin B\}$ and $\bar H=\textsf{Diag}(\eta^*)$, where $\eta_i^*=0$ for all $i\in B$ and $\eta_i^*=1$ for all $i\notin B$. Hence, the rank of $\bar H$ is $d-p$. H

ere, $\theta^*$ is to be recovered up to independent sign flips of its coordinates. In other words, the challenge is to recover the vector $(|\theta_1^*|,\ldots,|\theta_d^*|)^\top$. The $d-p$ nonzero entries of this vector are estimated at the rate $n^{-1/2}$ by the MLE, whereas the zero coordinates are only estimated at the rate $n^{-1/4}$.

\subsection{If $G$ is the group of coordinate cyclic shifts}

Denote by $R$ the elementary coordinate cyclic shift, i.e., for all $u=(u_1,\ldots,u_d)^\top\in\R^d$, $Ru=(u_2,u_3,\ldots,u_d,u_1)^\top$. Here, $G=\{I,R,R^2,\ldots,R^{d-1}\}$ is a cyclic group. Let $\theta^*\in\R^d$ and $p=\min\{k\geq 1: R^k\theta^*=\theta^*\}$. Then, $p$ is a divider of $d$ and $H=\{I,R^p,R^{2p},\ldots,R^{(d/p-1)p}\}$. Moreover, for all $u=(u_1,\ldots,u_d)^\top \in\R^d$, the coordinates of $\bar H u$ are given by $\DS (\bar H u)_j=\frac{d}{p}\sum_{k = j \mod p} u_k$, for all $j=1,\ldots,d$, and $\bar Hu=0$ if and only if $\DS \sum_{k = j\mod p} u_k =0$, for all $j=1,\ldots,d$. In particular, the dimension of the component of $\theta^*$ that can be estimated at the fast rate (i.e., the dimension of the range of $\bar H$) is $d-p$. Note that in that case, any constant vector $\theta^*$ is as hard to estimate via maximum likelihood as the null vector.

\subsection{If $G$ is the group of coordinate permutations}

Denote by $\Sig$ the symmetric group of order $d$ and for all $\sigma\in\Sig$, let $g_\sigma$ be the isometry that maps a vector $u=(u_1,\ldots,u_d)^\top\in\R^d$ to $g_\sigma(u)=(u_{\sigma(1)},u_{\sigma(2)},\ldots,u_{\sigma(d)})^\top$. Partition $[d]$ into sets $B_1,\ldots,B_p$, where $p$ is the cardinality of the set $\{\theta_1^*,\ldots,\theta_d^*\}$ and for each $k=1,\ldots,p$ and every $i,j\in B_k$, $\theta_i^*=\theta_j^*$.
Then, $\bar H$ is the set of all $g_\sigma$'s for which every orbit of $\sigma$ is contained in some $B_k$, for some $k\in [p]$, i.e., $\theta_i^*=\theta_{\sigma(i)}^*$, for all $i=1,\ldots,d$. In particular, for all $u\in\R^d$ and $j\in [d]$, $(\bar Hu)_j=\frac{1}{|B_{k_j}|}\sum_{i\in B_{k_j}}u_i$, where $k_j\in [p]$ is such that $j\in B_{k_j}$. Therefore, the rank of $\bar H$ is $d-p$. Again in this case, any constant vector $\theta^*$ is as hard to estimate via maximum likelihood as the null vector. Also note that in this case, estimating $\theta^*$ amounts to estimating the multiset $\{\theta_1^*,\ldots,\theta_d^*\}$ (by multiset, we mean the set where we keep track of repetitions).

\section{Conclusion}

In this work, we have exhibited two different pointwise rates for the estimation of the parameter of a mixture of Gaussian distributions with uniform weights, under the invariance of an isometry group action: $n^{-1/2}$ and $n^{-1/4}$. Even in the second regime, we have shown that some components of $\theta^*$ could still be estimated at the fast rate $n^{-1/2}$, and we have provided an algebraic description and a geometric interpretation of this fact, in terms of colliding modes of the population log-likelihood. These rates are consistent with the usual pointwise rates obtained known in the literature, even though here, we focused on parameter estimation (as opposed to distribution learning) with respect to the Euclidean loss.
 
As expected for general mixtures \cite{chen1995optimal}, the slow regime $n^{-1/4}$ occurs when the actual number of components in the mixture is strictly less than the number predicted by the model, here, $|G|$. In other words, for general mixtures, slower rates occur when the model is overparametrized. However, here, this analogy should be made carefully because the presence of symmetries in $\theta^*$ is not necessarily implying an overparametrization. 

The projection $\bar H$ depends on $\theta^*$. Therefore, even if Theorem \ref{ThmStats} states that some components of $\theta^*$ are estimated at the parametric rate $n^{-1/2}$ while the other components are estimated at the slower rate $n^{-1/4}$, the linear subspaces corresponding to these components are unknown. The problem of recovering $\bar H$ or, more generally, $H$, is somewhat equivalent to learning the symmetries of $\theta^*$. If one assumes that $\inf_{g\notin H}\|g\theta^*-\theta^*\|$ is bounded away from zero by some known constant, then $H$ can be recovered easily. In general, the estimation of $H$ is a more challenging problem, which we leave for further work.

\bibliographystyle{plain}
\bibliography{Biblio}

\appendix

\section{Proof of the main lemmas}

\subsection{Proof of Lemma \ref{MainLemma}}

The main key to this lemma is to note that for all $\theta\in \mathcal V$, 
\begin{equation} \label{mainlemma1}
\int_{\mathcal Y}L(y,\theta)\diff\mu(y)=1.
\end{equation} 
By dominated convergence, differentiating \eqref{mainlemma1} up to four times leads to the following identities, for all $\theta\in\mathcal V$ and $u\in\R^d$:

\begin{equation*}
\int_{\mathcal Y} u^\top \frac{\partial \log L}{\partial\theta}(y,\theta)L(y,\theta)\diff\mu(y) =0;
\end{equation*}

\begin{equation*}
\int_{\mathcal Y} u^\top \prt{\frac{\partial^2\log L}{\partial\theta\partial\theta^\top}(y,\theta)+\frac{\partial \log L}{\partial\theta}(y,\theta)\frac{\partial \log L}{\partial\theta^\top}(y,\theta)}u L(y,\theta)\diff\mu(y) =0;
\end{equation*}

\begin{align*}
	\int_{\mathcal Y} \Big [\partial_\theta^3(\log L)(y,\theta)(u,u,u) & +3\partial_\theta (\log L)(y,\theta)(u)\partial_\theta^2 (\log L)(y,\theta)(u,u) \nonumber \\ 
	& +\prt{\partial_\theta (\log L)(y,\theta)(u)}^3 \Big] L(y,\theta)\diff\mu(y) =0;
\end{align*}

\begin{align*}
	\int_{\mathcal Y} & \Big[\partial_\theta^4(\log L)(y,\theta)(u,u,u,u)+3\prt{\partial_\theta^2(\log L)(y,\theta)(u,u)}^2 \nonumber \\
	& +4\partial_\theta(\log L)(y,\theta)(u)\partial_\theta^3(\log L)(y,\theta)(u,u,u) \nonumber \\
	& +4\partial_\theta^2(\log L)(y,\theta)(u,u)\prt{\partial_\theta(\log L)(y,\theta)(u)}^2 \nonumber \\ 
	& +\prt{\partial_\theta(\log L)(y,\theta)(u)}^4       \Big]L(y,\theta)\diff\mu(y) =0.
\end{align*}

Taking $\theta=\theta^*$ in each of the above displays yields, for all $u\in\R^d$:

\begin{equation}\label{mainlemma2}
	\E_{\theta^*}\crc{u^\top \frac{\partial \log L}{\partial\theta}(Y,\theta^*)}=0;
\end{equation}

\begin{equation}\label{mainlemma3}
	\E_{\theta^*}\crc{u^\top \frac{\partial^2\log L}{\partial\theta\partial\theta^\top}(Y,\theta^*)u }=-\Var_{\theta^*}\crc{u^\top \frac{\partial \log L}{\partial\theta}(y,\theta^*)};
\end{equation}

\begin{align}\label{mainlemma5}
	& \E_{\theta^*} \crc{\partial_\theta^3(\log L)(Y,\theta^*)(u,u,u) } \nonumber \\
	& = -\E_{\theta^*}\crc{3\partial_\theta (\log L)(Y,\theta^*)(u)\partial_\theta^2 (\log L)(Y,\theta^*)(u,u)+\prt{\partial_\theta (\log L)(Y,\theta^*)(u)}^3};
\end{align}
and 

\begin{align}\label{mainlemma6}
	\E_{\theta^*} & \crc{\partial_\theta^4(\log L)(Y,\theta^*)(u,u,u,u) }\nonumber \\
	& = -3\E_{\theta^*}\crc{\prt{\partial_\theta^2 (\log L)(Y,\theta^*)(u,u)}^2} \nonumber \\
	& \quad -4\E_{\theta^*}\crc{\partial_\theta (\log L)(Y,\theta^*)(u)\partial_\theta^3 (\log L)(Y,\theta^*)(u,u,u)} \nonumber \\
	& \quad -4\E_{\theta^*}\crc{\partial_\theta^2 (\log L)(Y,\theta^*)(u,u)\prt{\partial_\theta (\log L)(Y,\theta^*)(u)}^2} \nonumber \\ 
	& \quad -\E_{\theta^*}\crc{\prt{\partial_\theta (\log L)(Y,\theta^*)(u)}^4}.
\end{align}

Now, by dominated convergence, for all $\theta\in\mathcal V$,
\begin{equation*}
	\frac{\partial\Psi}{\partial\theta}(\theta) = \E_{\theta^*}\crc{\frac{\partial\log L}{\partial\theta}(Y,\theta^*)} = \int_{\mathcal Y}\frac{\partial L}{\partial\theta}(y,\theta^*)\diff\mu(y) =0,
\end{equation*}
by \eqref{mainlemma2}, which proves the first part of the lemma. The second part is straightforward using \eqref{mainlemma3}. Now, let $u\in\R^d$ such that $\diff^2\Psi(\theta^*)(u,u)=0$. Then, by the second part of the lemma, $\DS \Var_{\theta^*}\crc{u^\top \frac{\partial \log L}{\partial \theta}(Y,\theta^*)}=0$, i.e., the random variable $\DS u^\top \frac{\partial \log L}{\partial \theta}(Y,\theta^*)$ must be constant, $\QQ_{\theta^*}$-almost surely. Since its expectation is zero, by the first part of the lemma, it must hold that $\DS u^\top \frac{\partial \log L}{\partial \theta}(Y,\theta^*)$, $\QQ_{\theta^*}$-almost surely. Plugging this into \eqref{mainlemma5} yields $\diff^3\Psi(\theta^*)(u,u,u)=0$, which is the third part of the lemma. Finally, in the same manner, \eqref{mainlemma6} yields, for all $u\in\R^d$ with $\diff^2\Psi(\theta^*)(u,u)=0$, that
\begin{equation} \label{mainlemma7}
	\E_{\theta^*}\crc{\partial_\theta^4(\log L)(Y,\theta^*)(u,u,u,u) } = -3\E_{\theta^*}\crc{\prt{\partial_\theta^2 (\log L)(Y,\theta^*)(u,u)}^2}.
\end{equation}
Since $\diff^2\Psi(\theta^*)(u,u)=0$ implies that $\DS \E_{\theta^*}\crc{\partial_\theta^2 (\log L)(Y,\theta^*)(u,u)}=0$, the right hand side of \eqref{mainlemma7} is equal to $\DS -3\Var_{\theta^*}\crc{\partial_\theta^2 (\log L)(Y,\theta^*)(u,u)}$. Hence, by dominated convergence,
\begin{align*}
	\diff^4\Psi(\theta^*)(u,u,u,u) & = -3\Var_{\theta^*}\crc{\partial_\theta^2 (\log L)(Y,\theta^*)(u,u)} \\ 
	& = -3\Var_{\theta^*}\crc{\frac{1}{L(Y,\theta^*)}u^\top\frac{\partial^2(\log L)}{\partial\theta\partial\theta^\top}(Y,\theta^*)u},
\end{align*}
using again the fact that $\DS \frac{\partial (\log L)}{\partial\theta^\top}(Y,\theta^*)u=0$, $\QQ_{\theta^*}$-almost surely. This ends the proof of Lemma \ref{MainLemma}.

\subsection{Proof of Lemma \ref{LemmaProj}}

It is easy to see that $H$ is a subgroup of $G$: Indeed, $I\in H$ and for all $g_1, g_2\in H$, $g_1\theta^*=\theta^*=g_2\theta^*$, yielding $g_1^{-1}g_2\theta^*= \theta^*$, i.e., $g_1^{-1}g_2\in H$. Hence, the map $h\to h^{-1}$ induces a bijection on $H$, and $\bar H$ might as well be written $\DS \bar H=\frac{1}{|H|}\sum_{h\in H}h^{-1}=\frac{1}{|H|}\sum_{h\in H}h^\top=\bar H^{\top}$, where we used the fact that $H$ is a set of isometries. Hence, $\bar H$ is symmetric. Note also that for all $h\in H$, $g\mapsto hg$ also induces a bijection on $H$, since $H$ is a subgroup. Therefore, $h\bar H=\bar H$, for all $h\in H$, yielding 
\begin{equation*}
	\bar H^2 = \frac{1}{|H|}\sum_{h\in H}h\bar H = \frac{1}{|H|}\sum_{h\in H}\bar H =\bar H. 
\end{equation*}
Hence, $\bar H$ is an orthogonal projection. Let $\mathcal F$ be the set of all $u\in\R^d$ such that $hu=u$, for all $h\in H$. It is clear that for all $u\in \mathcal F$, $\bar H u=u$, yielding that $\mathcal F$ is contained in the range of $\bar H$. Now, let $u$ be in the range of $\bar H$, i.e., such that $\bar Hu=u$. Then, for all $h\in H$, $hu=h\bar H u=\bar H u=u$, where we used that $h\bar H=\bar H$, for all $h\in H$. This ends the proof of the first part of the lemma. 

For the second part of the lemma, note that for all $S\in E$ and $g\in S$, $S=gH$, readily yielding $\bar S=g\bar H$.

Finally, for the last part of the lemma, let $S\in E$ and $g\in S$. Then, $\bar Sv=g\bar Hv=gv$ and $\bar Sw=g\bar Hw=0$.

\section{Intermediate lemmas}

\subsection{Proof of Lemma \ref{higherorderid}}

This lemma also comes from successive differentiations of \eqref{mainlemma1}, with respect to $\theta$ in the directions $v$ and $w$. For simplicity of the notation, we denote by $\pl{k}(u_1,\ldots,u_k)=\partial_\theta^k(\log L)(y,\theta)(u_1,\ldots,u_k)$, for all $k\geq 1$ and $u_1,\ldots,u_k\in\R^d$. Then, differentiating \eqref{mainlemma1}, first in the direction of $v$, then in the direction of $w$,  yields:
\begin{equation} \label{Truc1}
	\int \pll(v)L=0;
\end{equation}

\begin{equation} \label{Truc2}
	\int \crc{\pl{2}(v,w)+\pll(v)\pll(w)}L=0;
\end{equation}

\begin{equation} \label{Truc3}
	\int \crc{\pl{3}(v,w,w)+2\pl{2}(v,w)\pll(w)+\pll(v)\pl{2}(w,w)+\pll(v)\pll(w)^2}L=0;
\end{equation}

\begin{align} 
	\int & \Big[\pl{4}(v,w,w,w)+2\pl{3}(v,w,w)\pll(w)+3\pl{2}(v,w)\pl{2}(w,w) \nonumber \\
	& \quad +\pll(v)\pl{3}(w,w,w)+\pl{2}(v,w)\pll(w)^2+2\pll(v)\pl{2}(w,w)\pll(w)\Big]L=0. \label{Truc4}
\end{align}
Here, all the integrals should be understood with respect to the variable $y$, whose dependency is not included in our current notation, again, for the sake of simplicity.

Now, we show that $\pll(w), \pl{2}(v,w)$ and $\pl{3}(w,w,w)$ are all equal to zero for all $y\in\R^d$, and for $\theta=\theta^*$. The first statement of Lemma \ref{higherorderid} will then follow directly from \eqref{Truc3} and \eqref{Truc4}. Note that for all $S\in E$, $\bar Sw=0$ and for all $g\in S$, $gv=\bar Sv$.

\begin{equation*}
	\pll(w) = -\frac{\sum_{S\in E}(\bar S w)^\top(y-\bar S\theta^*)e^{-\frac{1}{2}\|y-\bar S\theta^*\|^2}}{\sum_{g\in G}e^{-\frac{1}{2}\|y-g\theta^*\|^2}} = 0.
\end{equation*}

\begin{align*}
	\pl{2}(v,w) & = \frac{\left\|\sum_{S\in E}(\bar S w)^\top(y-\bar S\theta^*)e^{-\frac{1}{2}\|y-\bar S\theta^*\|^2}\right\|^2}{\left(\sum_{g\in G}e^{-\frac{1}{2}\|y-g\theta^*\|^2}\right)^2} \\
	& \quad \quad \quad -\frac{\sum_{S\in E}\prt{|H|v^\top w-(\bar Sv)^\top (y-\bar S\theta^*)(y-\bar S\theta^*)^\top\bar Sw  }e^{-\frac{1}{2}\|y-\bar S\theta^*\|^2}}{\sum_{g\in G}e^{-\frac{1}{2}\|y-g\theta^*\|^2}} \\
	& = 0
\end{align*}
and again, it is easy to see that each term in $\pl{3}(w,w,w)$ contains a sum over $S\in E$ where $\bar Sw$ factorizes, yielding $\pl{3}(w,w,w)=0$.

\subsection{Proof of Lemma \ref{IntermediateLemma4}}

Denote by $F=\{v\in\R^d:\bar Hv=v\}$ and by $F^\perp=\{w\in\R^d:\bar Hw=0\}$ its orthogonal. We show that for all $v\in F$ and $w\in F^\perp$ with $v\neq 0$ and $w\neq 0$, 
\begin{equation} \label{RandomLabel1243}
\Var_{\theta^*}\crc{2v^\top\frac{\partial\log L}{\partial\theta}(Y,\theta^*)+w^\top\frac{\partial^2\log L}{\partial\theta\partial\theta^\top}(Y,\theta^*)w} \neq 0. 
\end{equation}
This will imply that for all $(v,w)\in F\times F^\perp$ with $v\neq 0$ and $w\neq 0$, 
\begin{equation}\label{RandomLabel123}
	\phi(v,w):=\left|\textsf{corr}_{\theta^*}\prt{2v^\top\frac{\partial\log L}{\partial\theta}(Y,\theta^*),w^\top\frac{\partial^2\log L}{\partial\theta\partial\theta^\top}(Y,\theta^*)w}\right| \neq 1.
\end{equation}
Indeed, if $\phi(v,w)=1$, then there must exist $\lambda\neq 0$ such that $w^\top\frac{\partial^2\log L}{\partial\theta\partial\theta^\top}(Y,\theta^*)w=\lambda v^\top\frac{\partial\log L}{\partial\theta}(Y,\theta^*)$ $\PP_{\theta^*}$-almost surely. This follows from the case of equality in Cauchy-Schwartz inequality, after noting that both $v^\top\frac{\partial\log L}{\partial\theta}(Y,\theta^*)$ and $w^\top\frac{\partial^2\log L}{\partial\theta\partial\theta^\top}(Y,\theta^*)w$ are not $\PP_{\theta^*}$-almost surely equal to zero, since by combining Theorem \ref{maintheorem} and Lemma \ref{MainLemma} (iii), the variance of the first random variable is nonzero, and the variance of a rescaled version of the second one is also nonzero. Therefore, the pair $(-\lambda v,w)$ violates \eqref{RandomLabel1243}. 

Let us denote by $\Sp$ the unit sphere in $\R^d$. Then, since the fonction $\phi$ defined above is continuous and $(F\cap\Sp)\times (F^\perp\cap \Sp)$ is a compact set, \eqref{RandomLabel123} implies that there exists $c\in [0,1)$ such that $\phi(v,w)\leq c$, for all $(v,w)\in (F\cap\Sp)\times (F^\perp\cap \Sp)$. Hence, by homogeneity, for all $(v,w)\in F\times F^\perp$ with $v\neq 0$ and $w\neq 0$, one still has $\phi(v,w)\leq c$ and
\begin{align*}
	& \left|\cov_{\theta^*}\prt{\crc{2v^\top\frac{\partial\log L}{\partial\theta}(Y,\theta^*),w^\top\frac{\partial^2\log L}{\partial\theta\partial\theta^\top}(Y,\theta^*)w}}\right| \\
	& \quad \quad \quad \leq c\sqrt{\Var_{\theta^*}\crc{2v^\top\frac{\partial\log L}{\partial\theta}(Y,\theta^*)}\Var_{\theta^*}\crc{w^\top\frac{\partial^2\log L}{\partial\theta\partial\theta^\top}(Y,\theta^*)w}} \\
	& \quad \quad \quad \leq c\Var_{\theta^*}\crc{2v^\top\frac{\partial\log L}{\partial\theta}(Y,\theta^*)} + c\Var_{\theta^*}\crc{w^\top\frac{\partial^2\log L}{\partial\theta\partial\theta^\top}(Y,\theta^*)w},
\end{align*}
yielding
\begin{align}
	\Var_{\theta^*}&\crc{2v^\top\frac{\partial\log L}{\partial\theta}(Y,\theta^*)+w^\top\frac{\partial^2\log L}{\partial\theta\partial\theta^\top}(Y,\theta^*)w} \nonumber \\
	& = \Var_{\theta^*}\crc{2v^\top\frac{\partial\log L}{\partial\theta}(Y,\theta^*)} +\Var_{\theta^*}\crc{w^\top\frac{\partial^2\log L}{\partial\theta\partial\theta^\top}(Y,\theta^*)w} \nonumber \\
	& \quad \quad \quad +2\cov_{\theta^*}\prt{\crc{2v^\top\frac{\partial\log L}{\partial\theta}(Y,\theta^*),w^\top\frac{\partial^2\log L}{\partial\theta\partial\theta^\top}(Y,\theta^*)w}} \nonumber \\
	& \geq (1-c)\Var_{\theta^*}\crc{2v^\top\frac{\partial\log L}{\partial\theta}(Y,\theta^*)} +(1-c)\Var_{\theta^*}\crc{w^\top\frac{\partial^2\log L}{\partial\theta\partial\theta^\top}(Y,\theta^*)w}. \label{RandomLabel121}
\end{align}
Now, by Theorem \ref{maintheorem} and by continuity, there exist positive constants $c_1$ and $c_2$ such that $\DS \Var_{\theta^*}\crc{2v^\top\frac{\partial\log L}{\partial\theta}(Y,\theta^*)}\geq c_1$ and $\DS \Var_{\theta^*}\crc{w^\top\frac{\partial^2\log L}{\partial\theta\partial\theta^\top}(Y,\theta^*)w}\geq c_2$, for all $v\in F\cap \Sp$ and $w\in F^\perp\cap\Sp$. Therefore, \eqref{RandomLabel121} yields the desired result, by homogeneity. Thus, what remains to be proved is \eqref{RandomLabel1243}. For that purpose, let $v\in F$ and $w\in F^\perp$ such that \eqref{RandomLabel1243} does not hold. Let us show that necessarily, $v=w=0$. First, note that it must hold that $\DS 2v^\top\frac{\partial\log L}{\partial\theta}(Y,\theta^*)+w^\top\frac{\partial^2\log L}{\partial\theta\partial\theta^\top}(Y,\theta^*)w$ is constant $\PP_{\theta^*}$-almost surely. Since its expectation is zero (the first term has expectation $\DS 2\diff\Psi(\theta^*)(v)$ which is zero since $\theta^*$ is a local maximum of $\Psi$ and the second term has expectation $\DS \diff^2\Psi(\theta^*)(w,w)$ which is zero by Theorem \ref{maintheorem}), it must hold that
\begin{equation} \label{RandomLabel6839}
2v^\top\frac{\partial\log L}{\partial\theta}(y,\theta^*)+w^\top\frac{\partial^2\log L}{\partial\theta\partial\theta^\top}(y,\theta^*)w=0, \quad \forall y\in\R^d.
\end{equation}
 
\subparagraph{Step 1: Computing $\DS v^\top\frac{\partial\log L}{\partial\theta}(y,\theta^*)$}

Recall that for all $y\in\R^d$ and $\theta\in\R^d$,
\begin{equation} \label{RecallFormula}
	\log L(y,\theta)=-\log\prt{(2\pi)^{d/2}|G|}+\log\sum_{g\in G}e^{-\frac{1}{2}\|y-g\theta\|^2}.
\end{equation}

Differentiating \eqref{RecallFormula} with respect to $\theta$ in the direction of $v$, and plugging $\theta=\theta^*$ yields
\begin{align*}
	v^\top\frac{\partial\log L}{\partial\theta}(y,\theta^*) & = -\frac{\sum_{g\in G}e^{-\frac{1}{2}\|y-g\theta^*\|^2}v^\top g^\top(y-g\theta^*)}{\sum_{g\in G}e^{-\frac{1}{2}\|y-g\theta^*\|^2}} \\
	& = -\frac{|H|\sum_{S\in E}e^{-\frac{1}{2}\|y-\bar S\theta^*\|^2}(\bar Sv)^\top (y-\bar S\theta^*)}{\sum_{g\in G}e^{-\frac{1}{2}\|y-g\theta^*\|^2}} \\
	& = -\frac{|H|\sum_{S\in E}e^{-\frac{1}{2}\|y-\bar S\theta^*\|^2}v^\top \bar S^\top y}{\sum_{g\in G}e^{-\frac{1}{2}\|y-g\theta^*\|^2}} -v^\top\theta^*, 
\end{align*}
where we used that $gv=\bar Sv$ for all $S\in E$ and all $g\in S$ in the second equality and that $\bar S^\top\bar S\theta^*=\theta^*$, for all $S\in E$ in the third equality.

\subparagraph{Step 2: Computing $w^\top\frac{\partial^2\log L}{\partial\theta\partial\theta^\top}(y,\theta^*)w$}

Now, differentiating \eqref{RecallFormula} twice with respect to $\theta$ in the direction of $w$, and plugging $\theta=\theta^*$ yields

\begin{align}
w^\top\frac{\partial^2\log L}{\partial\theta\partial\theta^\top}(y,\theta^*)w & = \frac{\left \|\sum_{g\in G}e^{-\frac{1}{2}\|y-g\theta^*\|^2}(y-g\theta^*)^\top gw \right\|^2}{\prt{\sum_{g\in G}e^{-\frac{1}{2}\|y-g\theta^*\|^2}}^2} \nonumber \\
& \quad -\frac{\sum_{g\in G}e^{-\frac{1}{2}\|y-g\theta^*\|^2} \prt{-\|w\|^2+w^\top g^\top(y-g\theta^*)(y-g\theta^*)^\top gw}}{\sum_{g\in G}e^{-\frac{1}{2}\|y-g\theta^*\|^2}}. \label{RandomLabel9837}
\end{align}
In the first term of the right hand side of \eqref{RandomLabel9837}, the sum inside the squared norm can be rewritten as $\DS|H|\sum_{S\in E}e^{-\frac{1}{2}\|y-\bar S\theta^*\|^2}(y-\bar S\theta^*)^\top \bar Sw$, which is zero, since for any $S\in E$ and any arbitrary $g\in S$, one can write $\bar Sw=g\bar Hw=0$. Thus, after trivial simplifications,
\begin{equation}
w^\top\frac{\partial^2\log L}{\partial\theta\partial\theta^\top}(y,\theta^*)w = \|w\|^2-\frac{\sum_{g\in G}e^{-\frac{1}{2}\|y-g\theta^*\|^2}w^\top g^\top yy^\top gw }{\sum_{g\in G}e^{-\frac{1}{2}\|y-g\theta^*\|^2}}. \label{RandomLabel9}
\end{equation}

\subparagraph{Step 3: Concluding}
Therefore, \eqref{RandomLabel6839} implies that, for all $y\in\R^d$, 
\begin{equation} \label{RandomLabel39481}
	-v^\top\theta^*+\|w\|^2-\frac{|H|\sum_{S\in E}e^{-\frac{1}{2}\|y-\bar S\theta^*\|^2}v^\top \bar S^\top y}{\sum_{g\in G}e^{-\frac{1}{2}\|y-g\theta^*\|^2}} -\frac{\sum_{g\in G}e^{-\frac{1}{2}\|y-g\theta^*\|^2}w^\top g^\top yy^\top gw }{\sum_{g\in G}e^{-\frac{1}{2}\|y-g\theta^*\|^2}}=0.
\end{equation}
Taking $y=0$ yields $v^\top\theta^*=\|w\|^2$, hence, \eqref{RandomLabel39481} becomes
\begin{equation} \label{RandomLabe81}
	\sum_{S\in E}e^{-\frac{1}{2}\|y-\bar S\theta^*\|^2}\prt{|H|v^\top \bar S^\top y -\sum_{g\in S}e^{-\frac{1}{2}\|y-g\theta^*\|^2}(w^\top g^\top y)^2} =0, \forall y\in\R^d.
\end{equation}
Now, using a similar argument as in the proof of Theorem \ref{maintheorem}, this implies that for all $S\in E$ and $y\in\R^d$, $\DS |H|v^\top \bar S^\top y -\sum_{g\in S}e^{-\frac{1}{2}\|y-g\theta^*\|^2}(w^\top g^\top y)^2=0$. Hence, both the linear and the quadratic terms in $y$ need to be zero, implying $v=w=0$.

\subsection{Proof of Lemma \ref{TaylorExp}}

\begin{proof} 
With the same computations as in the proof of Lemma \ref{IntermediateLemma4} below, one can show that for all $y\in\R^d$ and for $v,w\in\R^d$ with $\bar Hv=v$ and $\bar H w=0$,
\begin{equation*}
	w^\top\frac{\partial\log L}{\partial\theta}(y,\theta^*)=0;
\end{equation*}

\begin{equation*}
	\left|(v+w)^\top\frac{\partial^2\log L}{\partial\theta\partial\theta^\top}(y,\theta^*)(v+w)\right| \leq c\|y\|^2 (\|v\|^2+\|w\|^2);
\end{equation*}

\begin{equation*}
	\sup_{0\leq t\leq 1}\partial_\theta^3(\log L)(tv+tw,tv+tw,tv+tw)|\leq c(\|y\|+\|y^2\|+\|y\|^3)(\|v\|^3+\|w\|^3),
\end{equation*}
where $c>0$ is some positive constant. 

Now, let $y\in\R^d$ and $\theta\in\Theta$. Let $g_0\in G$ such that $\DS \|g_0\theta-\theta^*\|=\min_{g\in G}\|g\theta-\theta^*\|$ and write $g_0\theta-\theta^*=v+w$, where $v,w\in\R^d$ are such that $\bar Hv=v$ and $\bar Hw=0$. Then, if $\|v\|^2+\|w\|^4\leq \delta^2$, a Taylor expansion yields
\begin{equation*}
	\left|\log L(y,\theta)-\log L(y,\theta^*)\right|\leq c(1+\|y\|+\|y\|^2+\|y\|^3)\delta,
\end{equation*}
as long as $\delta$ is small enough, independently of $y$. This ends the proof of the lemma, with $F(y)=c(1+\|y\|+\|y\|^2+\|y\|^3), y\in\R^d$.
\end{proof}

\end{document}